\documentclass[reqno,12pt]{amsart}
\usepackage[colorlinks=true, linkcolor=blue, citecolor=blue]{hyperref}

\usepackage{amssymb}
\usepackage{amsmath, graphicx, rotating}
\usepackage{color}
\usepackage{soul}
\usepackage[dvipsnames]{xcolor}
 \usepackage{cite}

\usepackage{ifthen}
\usepackage{xkeyval}
\usepackage{todonotes}
\usepackage[T1]{fontenc}
\usepackage{lmodern}
\usepackage[english]{babel}

\usepackage{ upgreek }
\usepackage{stmaryrd}
\SetSymbolFont{stmry}{bold}{U}{stmry}{m}{n}
\usepackage{amsthm}
\usepackage{float}

\usepackage{ bbm }
\usepackage{ stmaryrd }
\usepackage{ mathrsfs }
\usepackage{ frcursive }
\usepackage{ comment }

\usepackage{pgf, tikz}
\usetikzlibrary{shapes}
\usepackage{varioref}
\usepackage{enumitem}

\definecolor{rouge}{rgb}{0.7,0.00,0.00}
\definecolor{vert}{rgb}{0.00,0.5,0.00}
\definecolor{bleu}{rgb}{0.00,0.00,0.8}
\usepackage[margin=1in]{geometry}
\newtheorem{theorem}{Theorem}[section]
\newtheorem*{theorem*}{Theorem}
\newtheorem{lemma}[theorem]{Lemma}

\newtheorem{proposition}[theorem]{Proposition}

\labelformat{hypothesis}{\textbf{M\kern-0.1mm#1}}

\newtheorem{assumption}{Assumption}

\newtheorem{assumptionA}{A\kern-0.1mm}
\labelformat{assumptionA}{\textbf{A\kern-0.1mm#1}}

\theoremstyle{definition}

\newtheorem{remark}[theorem]{Remark}

\def \eref#1{\hbox{(\ref{#1})}}

\numberwithin{equation}{section}

\def\geq{\geqslant}
\def\leq{\leqslant}

\def\RR{\mathbb{R}}
\def\PP{\mathbb{P}}
\def\EE{\mathbb{E}}
\def\NN{\mathbb{N}}

\def\vare{{\varepsilon}}
\def \eref#1{\hbox{(\ref{#1})}}

\def\EE{\mathbb{ E}}

\DeclareMathOperator{\e}{e}

\begin{document}

\title[Averaging principle for slow-fast singular SPDEs driven by $\alpha$-stable process]
{Strong averaging principle for a class of slow-fast singular SPDEs driven by $\alpha$-stable process}

\author{Xiaobin Sun}
\curraddr[Sun, X.]{ School of Mathematics and Statistics and Research Institute of Mathematical Science (RIMS), Jiangsu Normal University, Xuzhou, 221116, China}
\email{xbsun@jsnu.edu.cn}

\author{Huilian Xia}
\curraddr[Xia, H.]{ School of Mathematics and Statistics and Research Institute of Mathematical Science (RIMS), Jiangsu Normal University, Xuzhou, 221116, China}
\email{huilxia@jsnu.edu.cn}

\author{Yingchao Xie}
\curraddr[Xie, Y.]{ School of Mathematics and Statistics and Research Institute of Mathematical Science (RIMS), Jiangsu Normal University, Xuzhou, 221116, China}
\email{ycxie@jsnu.edu.cn}

\author{Xingcheng Zhou}
\curraddr[Zhou, X.]{ School of Mathematics and Statistics and Research Institute of Mathematical Science (RIMS), Jiangsu Normal University, Xuzhou, 221116, China}
\email{xczhou@jsnu.edu.cn}

\begin{abstract}
In this paper, the strong averaging principle is researched for a class of H\"{o}lder continuous drift slow-fast SPDEs with $\alpha$-stable process by the Zvonkin's transformation and the classical Khasminkii's time discretization method. As applications, an example is also provided to explain our result.
\end{abstract}
\date{}

\maketitle
\section{Introduction}

The slow-fast system (also called two-time scales system) can be described through a coupled system, which corresponds to the "slow" and "fast" components.
This kind of system has appeared in many fields, such as the nonlinear oscillations, chemical kinetics, biology, climate dynamics, see e.g. \cite{BR, EE, HKW, Ku}.
The averaging principle for slow-fast system shows the asymptotic behavior of the slow component when the scale $\vare$ goes to $0$. The pioneering works are attributed to the Bogoliubov and Mitropolsky \cite{BM} for the deterministic systems and Khasminskii \cite{K1} for stochastic differential equations (SDEs for short). Recently, Cerrai and
Freidlin \cite{C1,CF} first study the averaging principle for stochastic partial differential equations ( SPDEs for short), i.e., considering a class of two-time scale stochastic partial differential equations on a bounded domain $D\subset \RR^{d}$ ($d>1$) with Dirichlet boundary condition,
\begin{eqnarray*}\left\{\begin{array}{l}
\displaystyle
d u_{\vare}(t) = \left[A_1 u_{\vare}(t)+B_1(u_{\vare}(t),v_{\vare}(t))\right]dt+G_1(u_{\vare}(t),v_{\vare}(t)) dw^{Q_1}_t,\quad u_{\vare}(0)=x\in L^{2}(D), \nonumber\\
\displaystyle d v_{\vare}(t) =\frac{1}{\vare}\left[A_2 v_{\vare}(t)+B_2(u_{\vare}(t),v_{\vare}(t))\right]dt+ \frac{1}{\sqrt{\vare}}G_2(u_{\vare}(t),v_{\vare}(t))d w^{Q_2}_t,\quad v_{\vare}(0)=y\in L^{2}(D),
\end{array}\right.\nonumber
\end{eqnarray*}
where $w^{Q_1}_t$ and $w^{Q_2}_t$ are Gaussian noises which are white in time and colored in space, in the case of space dimension $d>1$, with covariances operators $Q_1$ and $Q_2$. The operators $A_1$ and $A_2$ are second order uniformly elliptic operators. The coefficients $B_1,B_2:L^{2}(D) \times L^{2}(D) \rightarrow L^{2}(D)$, $G_1,G_2:L^{2}(D) \times L^{2}(D) \rightarrow \mathcal{L}(L^{2}(D);L^1(D))\cap\mathcal{L}(L^{\infty}(D); L^2(D))$ are Lipschitz continuous. Moreover, there exist two Lipschitz-continuous mappings $\bar{B_1}: L^2(D)\rightarrow L^2(D)$ and $\bar{G}: L^2(D)\rightarrow \mathcal{L}(L^{\infty}(D); L^2(D))$ such that for any $T>0$, $t\geq 0$ and $x,y\in L^2(D)$,
\begin{eqnarray*}
&&\EE\left|\frac{1}{T}\int^{t+T}_{t} \langle B_1(x, v^{x,y}(s)), h\rangle_{L^{2}(D)}ds-\langle\bar{B_1}(x),h\rangle_{L^{2}(D)}\right|\\
&\leq &\alpha(T)(1+|x|_{L^2(D)}+|y|_{L^2(D)})|h|_{L^2(D)}
\end{eqnarray*}
holds for any $ h\in L^2(D)$ and for any $h,k\in L^{\infty}(D)$,
\begin{eqnarray*}
&&\left|\frac{1}{T}\int^{t+T}_{t} \EE \langle G_1(x, v^{x,y}(s))h, G_1(x, v^{x,y}(s))k\rangle_{L^2(D)}ds-\langle \bar{G}_1(x)h, \bar{G}_1(x)k\rangle_{L^2(D)}\right|\\
&\leq& \alpha(T)(1+|x|^2+|y|^2)|h|_{L^{\infty}(D)}|k|_{L^{\infty}(D)}
\end{eqnarray*}
for some function $\alpha(T)$ vanishing as $T$ goes to infinity, where $v^{x,y}_t$ is the unique solution to the corresponding frozen equation,
$$
d v^{x,y}(t) =A_1v^{x,y}(t) dt+B_1(x, v^{x,y}_{t})dt+G_1(x, v^{x,y}_{t})d w^{Q_2}_t,\quad v^{x,y}_0=y.
$$
Then Cerrai \cite{C1} proves that the slow component $X^{\vare}$ converges $\bar{X}$ in distribution, i.e.,
$$
\mathcal{L}(u^{\vare})\rightarrow \mathcal{L}(\bar{u}) \quad \text{in} \quad C([0,T];L^2(D)),
$$
where $\bar{u}_t$ is the unique solution of the corresponding averaged equation,
$$
d \bar{u}(t) =A_1\bar{u}(t) dt+\bar{B}_1(\bar{u}(t))dt+\bar{G}_1(\bar{u}(t))d w^{Q_1}_t.
$$
For more results about the averaging principle for stochastic systems, please see e.g. \cite{Gi, L, LRSX1, RSX2, PXW2020, XPW} for SDEs  and see e.g. \cite{B1,C1,DSXZ,FLL,GP2,PXW2017,PXY,WRD12} for SPDEs.
\vspace{0.1cm}

However, the coupled drift coefficients $B_1$ and $B_2$ always satisfy at least local Lipshitz continuous, whereas the stochastic systems with irregular coefficients have not been studied much. As far as we know, in the case of SDEs, Veretennikov \cite{V0} studies the averaging principle for SDEs under the assumptions
that the drift coefficient of slow equation is only bounded and measurable with respect to
slow variable, and all the other coefficients are global Lipschitz continuous. R\"{o}ckner et al. \cite{RSX} study the strong and weak convergence in the averaging principle for SDEs
with H\"{o}lder coefficients drift. Also see e.g. \cite{PV1,PV2} for the study of diffusion approximations for SDEs with singular coefficients. In the case of SPDEs, Sun et al. \cite{SXX1} study the averaging principle for SPDEs with bounded H\"{o}lder continuous coefficients driven by Wiener noise. R\"{o}ckner et al. \cite{RXY} study the averaging principle for SPDEs when the coefficients are H\"{o}lder continuous with respect to the second variable.

\vspace{0.1cm}
In the past ten years, the SPDEs driven by cylindrical $\alpha$-stable noises have attracted much attention by many scholars, see e.g. \cite{DXZ,DXZ1,PZ,Xu}.
Considering the cylindrical $\alpha$-stable noises has theoretically meaningful, for instance such processes can be used to model systems with heavy tails in physics.
Recently, we have proved the well-posedness of a class of SPDEs driven by $\alpha$-stable process with bounded H\"{o}lder continuous in \cite{SXX0}. Thus, the main purpose of this paper is to further study the averaging principle for such kind of SPDEs, i.e, considering the following stochastic system in a Hilbert space $H$:
\begin{equation}\left\{\begin{array}{l}\label{Main equation}
\displaystyle
dX^{\varepsilon}_t=\left[A X^{\varepsilon}_t
+B(X^{\varepsilon}_t, Y^{\varepsilon}_t)\right]dt+d L_{t},\vspace{2mm}\\
\displaystyle dY^{\varepsilon}_t=\frac{1}{\varepsilon}\left[AY^{\varepsilon}_t+F(X^{\varepsilon}_t, Y^{\varepsilon}_t)\right]dt
+\frac{1}{\varepsilon^{1/\alpha}}d Z_{t},\\
\end{array}\right.
\end{equation}
where $\varepsilon>0$ is a small parameter describing the ratio of the time scales of the slow component $X^{\varepsilon}_t$
and the fast component $Y^{\varepsilon}_t$, $A: \mathscr{D}(A)\to H$ is the infinitesimal generator of a linear
strongly continuous semigroup $\{e^{tA}\}_{t\geq 0}$, $B, F: H\times H\rightarrow H$, which are bounded and H\"{o}lder continuous functions,
$\{L_t\}_{t\geq 0}$ and $\{Z_t\}_{t\geq 0}$ are $H$-valued mutually independent cylindrical $\alpha$-stable processes with $\alpha\in(1,2)$ defined on a complete filtered probability space
$(\Omega,\mathscr{F},\{\mathscr{F}_{t}\}_{t\geq0},\mathbb{P})$. Under some proper conditions, we will show that for any $T>0$ and $p\geq 2$,
$$\lim_{\vare\rightarrow 0}\mathbb{E} \left(\sup_{t\in[0,T]}|X_{t}^{\vare}-\bar{X}_{t}|^{p} \right)=0,$$
where $\bar{X}_t$ is the solution to the corresponding averaged equation (see \eref{3.1}  below).

\vspace{0.1cm}
There are some existing results about the averaging principle for slow-fast SPDEs driven by $\alpha$-stable process.
For instance, Bao et al. \cite{BYY} prove the strong averaging principle for two-time scale SPDEs driven by $\alpha$-stable noise.
The first author of the paper and his collaborators prove the strong averaging principle for stochastic Ginzburg-Landau equation and stochastic Burgers equations
driven by $\alpha$-stable processes in \cite{SZ} and \cite{CSS} respectively. However, the above-mentioned results require that the coupled coefficients $B$ and $F$ satisfy
the global Lipschitz continuous. Hence, many techniques used there do not work in the case that the coupled coefficients $B$ and $F$ are only H\"{o}lder continuous.

\vspace{0.1cm}
As known to all, a powerful technique of changing the singular coefficients to regular
ones is the well-known Zvonkin's transformation, which is widely used to study the strong well-posedness for S(P)DEs with singular coefficients (see e.g. \cite{DF, DFPR,KR} ).
Thus we will use the Zvonkin's transformation and combine the classical Khasminskii's time discretization to prove our main result.
In comparison to the Wiener noise considered in existing works \cite{V0,RXY,SXX1}, the $\alpha$-stable noise considered here will cause some difficulties,
such as the solution $(X_{t}^{\vare}, Y_{t}^{\vare})$ has only $p$-th moment for $p\in (0,\alpha)$, hence some methods developed there do not work in this situation.
\vspace{0.1cm}

The rest of the paper is organized as follows. In Section 2, we give some notations and the assumptions. In section 3, we give some a-priori estimates of the solution $(X^{\varepsilon},Y^{\varepsilon})$
and study the frozen and averaged equations.  In section 4, the Zvonkin's transformation and classical Khasminskii's time discretization method are presented first, then we state our main result and give its detailed proofs. In Section 5, we will give an example to illustrate the applicability of our result.

Throughout this paper, $C$, $C_p$, $C_T$ and $C_{p,T}$ are positive constants which may change from line to line, where the subscript $p, T$ are used to emphasize that the constant only depends
on the parameters $p, T$.

\vspace{0.1cm}

\section{Notations and assumptions} \label{Sec Main Result}

Let $H$ be a Hilbert space, whose inner product and norm are denoted by $\langle\cdot,\cdot\rangle$ and $|\cdot|$ respectively.

For a given $\beta\in (0,1]$, $C^{\beta}_b(H,H)$ denote the space of all bounded and H\"{o}lder continuous functions $G(x): H\rightarrow H$ with index $\beta$, and its norm is defined
$$
\|G\|_{C^{\beta}_b}:=\|G(x)\|_{\infty}+\|G\|_{C^{\beta}}<\infty,
$$
where $\|G(x)\|_{\infty}:=\sup_{x\in H}|G(x)|$ and $\|G\|_{C^{\beta}}:=\sup_{x\neq y\in H}\frac{|G(x)-G(y)|}{|x-y|^{\beta}}$. Similarly, given $\beta\in (1,2]$, the space $C^{\beta}_b(H,H)$
represents the function space satisfying
$$\|G\|_{C^{\beta}_b}:=\|G\|_{C^1_b}+\sup_{x\neq y\in H}\frac{\|DG(x)-DG(y)\|}{|x-y|^{\beta-1}}<\infty,$$
where $\|\cdot\|$ denotes the operator norm.

The cylindrical $\alpha$-stable process $\{L_t\}_{t\geq 0}$ and $\{Z_t\}_{t\geq 0}$ in (\ref{Main equation}) are denoted by
\begin{eqnarray*}
L_t:=\sum_{n\geq1}\beta_nL^n_te_n\quad \text{and}\quad Z_t:=\sum_{n\geq1}\gamma_nZ^n_te_n,\quad t\geq0,
\end{eqnarray*}
where $\{\beta_n\}_{n\geq1}$ and $\{\gamma_n\}_{n\geq1}$ are two given sequences of positive numbers, $\{e_n\}_{n\geq1}$ is a complete orthonormal basis of $H$. $\{L^n_t\}_{n\geq1}$ and $\{Z^n_t\}_{n\geq1}$ are two sequences of independent one dimensional rotationally symmetric $\alpha$-stable processes with $\alpha\in(1,2)$, i.e., the L\'{e}vy measure of $L^n_t$ and $Z^n_t$ is given by
\begin{eqnarray*}
\nu(dz)=\frac{c_{\alpha}}{|z|^{1+\alpha}}dz,
\end{eqnarray*}
where $c_{\alpha}>0$ is a constant.  We also assume that $\{L^n_t\}_{n\geq1}$ and $\{Z^n_t\}_{n\geq1}$ are independent.

By L\'{e}vy-It\^{o}'s decomposition, we can get
$$L^n_t=\int_{|z|\leq c}z\tilde{N}^{1,n}(t,dz)+\int_{|z|>c}z{N}^{1,n}(t,dz),$$
$$Z^n_t=\int_{|z|\leq c}z\tilde{N}^{2,n}(t,dz)+\int_{|z|>c}z{N}^{2,n}(t,dz),$$
where $c$ can be chosen for any positive constant. For any $n\in \NN_+$, denote the Poisson random measures as follows:
$${N}^{1,n}([0,t],\Gamma):=\sum_{0\leq s\leq t}1_{\Gamma}(L^n_s-L^n_{s^-}), \quad t>0, \Gamma\in \mathscr{B}(\RR\setminus\{0\}),$$
$${N}^{2,n}([0,t],\Gamma):=\sum_{0\leq s\leq t}1_{\Gamma}(Z^n_s-Z^n_{s^-}), \quad t>0, \Gamma\in \mathscr{B}(\RR\setminus\{0\}).$$
Then its corresponding compensated Poisson random measures are given by
\begin{eqnarray*}
\tilde{N}^{i,n}([0,t],\Gamma):= {N}^{i,n}([0,t],\Gamma)-t\nu(\Gamma),\quad i=1,2.
\end{eqnarray*}

\vspace{0.3cm}
We give the following assumptions throughout this paper:
\begin{assumptionA}\label{A1}
$A$ is a selfadjoint operator, $Ae_n=-\lambda_n e_n$ with $\lambda_n>0$ and $\lambda_n\uparrow \infty$, as $n\uparrow \infty$.
Where $\{e_n\}_{n\geq1}\subset \mathscr{D}(A)$ is a complete orthonormal basis of $H$.
\end{assumptionA}

\begin{assumptionA}\label{A2} There exists $\eta\in(0,1)$ such that
$$\sum_{n\geq1}\beta^{\alpha}_n<\infty,\quad \sum_{n\geq1}\gamma^{\alpha}_n<\infty,\quad \sum_{n\geq1}\lambda^{-1}_n<\infty.$$
\end{assumptionA}

\begin{assumptionA}\label{A3}
There exists a constant $\gamma \in (1,\alpha] $ such that for $\gamma^{'}<\gamma$ and $\lambda>0$, we have
$$\int^{\infty}_0 e^{-\lambda t}\Lambda_t^{\gamma^{'}}dt<\infty,$$
where $\Lambda_t:=\max\{\Lambda_1(t),\Lambda_2(t)\}$ with
$$\Lambda_1(t) := \sup\limits_{n\geq1}\frac{e^{-\lambda_nt}\lambda^{1/\alpha}_n}{\beta_n},\quad \Lambda_2(t) := \sup\limits_{n\geq1}\frac{e^{-\lambda_nt}\lambda^{1/\alpha}_n}{\gamma_n}.$$
Moreover, there exists $\kappa_1\in(0,1/2)$ such that for any $\lambda>0$,
\begin{eqnarray}
\int^{\infty}_0 e^{-\lambda t}\Lambda_{3,\kappa_1}(t)dt<\infty,\label{A31}
\end{eqnarray}
where $\Lambda_{3,\kappa_1}(t) := \sup\limits_{n\geq1}\frac{e^{-\lambda_nt}\lambda^{\kappa_1+1/\alpha}_n}{\beta_n}$.
\end{assumptionA}

\begin{assumptionA}\label{A4}
$B, F: H\times H\rightarrow H$ are bounded and measurable functions. In addition, there exist constants $\eta_1,\eta_2,\eta_3\in (0,1)$ satisfying $\eta_1\wedge(\eta_2\eta_3)\in (1+\alpha/2-\gamma,1)$ such that for any $x_1,x_2,y_1,y_2\in H$,
$$|B(x_1, y_1)-B(x_2, y_2)|\leq C\left(|x_1-x_2|^{\eta_1} + |y_1-y_2|^{\eta_2}\right),$$
$$|F(x_1, y_1)-F(x_2, y_2)|\leq C|x_1-x_2|^{\eta_3}+ L_F|y_1-y_2|,$$
where $L_F$ is a constant which satisfies $\lambda_{1}-L_{F}>0.$
\end{assumptionA}

\begin{remark}
It is worthy to point out that $F(x,y)$ is Lipschitz continuous with respect to $y$ uniformly for $x$ in \ref{A4}, which seems a little strong, however this condition is used to prove the H\"{o}lder continuous of the averaged coefficients $\bar{B}$. In other words, if $F(x,y)$ is assumed H\"{o}lder continuous with respect to $y$ by index $\gamma\in (0,1)$, it is not easy to prove that $\bar{B}$ is still H\"{o}lder continuous, indeed it is expected that the H\"{o}lder index of $\bar{B}$ should be $\eta_1\wedge \eta_3$ and maybe the technique of Poisson equation is helpful to prove it, see \cite{RXY}.
\end{remark}

\begin{remark}\label{remark 2.2}
Refer to \cite[(4.12)]{PZ}, if $\sum_{n\geq1}\frac{\beta^{\alpha}_n}{\lambda^{1-\alpha\theta/2}_n}<\infty$ for some $\theta>0$, then for any $0<p<\alpha$, we have
\begin{eqnarray}\label{Ls}
\sup_{t\geq 0}\mathbb{E}\Big\|\int^t_0 e^{(t-s)A}dL_s\Big\|^p_{\theta}\Big]\leq C_{p}\Big(\sum_{n\geq1}\frac{\beta^{\alpha}_n}{\lambda^{1-\alpha\theta/2}_n}\Big)^{p/\alpha},
\end{eqnarray}
where $\|\cdot\|_{\theta}$ is the norm in space $H^{\theta}$, whose detailed definition is given below. Similar, if $\sum_{n\geq1}\frac{\gamma^{\alpha}_n}{\lambda_n}<\infty$, then for any $0<p<\alpha$, we have
\begin{eqnarray}\label{Zs}
\sup_{t\geq 0}\mathbb{E}\Big|\int_0^te^{(t-s)A}dZ_s\Big|^p\leq C_{p}\Big(\sum_{n\geq1}\frac{\gamma^{\alpha}_n}{\lambda_n}\Big)^{p/\alpha}.
\end{eqnarray}
\end{remark}

\vspace{0.3cm}

For any $s\geq 0$, we define
 $$H^s:=\mathscr{D}((-A)^{s/2}):=\left\{u=\sum^{\infty}_{n=1}u_n e_n: u_n=\langle u,e_n\rangle\in \mathbb{R},~\sum^{\infty}_{n=1}\lambda_n^{s}u_n^2<\infty\right\}.$$
Putting
 $$(-A)^{s/2}u:=\sum^{\infty}_{n=1}\lambda_n^{s/2} u_n e_n,\quad u\in\mathscr{D}((-A)^{s/2}),$$
with the associated norm
\begin{eqnarray*}
\|u\|^2_{s}:=|(-A)^{s/2}u|^2=\sum^{\infty}_{n=1}\lambda_n^{s} u^2_n.
\end{eqnarray*}
It is not difficult to find that $H^{0}=H$. It's easy to check that for any $\theta>0$, there exists a constant $C_{\theta}>0$ such that
\begin{eqnarray}
|e^{tA}x|\leq e^{-\lambda_1 t}|x|, \quad x\in H, t\geq 0;\label{SP1}\end{eqnarray}
\begin{eqnarray}\|e^{tA}x\|_\theta\leq C_{\theta} t^{-\theta/2}|x|, \quad x\in H, t>0;\label{SP2}\end{eqnarray}
\begin{eqnarray}| e^{At}x-x | \leq C_{\theta}t^{\theta/2}\|x\|_{\theta},\quad x\in \mathscr{D}((-A)^{\theta/2}), t\geq 0.\label{SP3}\end{eqnarray}

\section{A-priori estimates, frozen and averaged equations} \label{S3}

This section is divided into two parts. We study the well-posedness of system \eref{Main equation} and give some a-priori estimates of the solution $(X^{\varepsilon},Y^{\varepsilon})$ in subsection 3.1. Then the frozen equation and the averaged equation are studied in subsection 3.2. Note that we always assume \ref{A1}-\ref{A4} hold and fix the initial values $(x,y)\in H\times H$ throughout this paper.
\vspace{0.3cm}

\subsection {Well-posedness and a-priori estimates of \texorpdfstring{$(X^{\varepsilon}_t, Y^{\varepsilon}_t)$}{Lg}}

\begin{lemma} \label{PMY}
The system \eqref{Main equation} has a unique mild solution $(X^{\varepsilon},Y^{\varepsilon})$, i.e., $\PP$-a.s.,
\begin{eqnarray*}
\left\{\begin{array}{l}
\displaystyle
X^{\varepsilon}_t=e^{tA}x+\int^t_0e^{(t-s)A}B(X^{\varepsilon}_s, Y^{\varepsilon}_s)ds+\int^t_0 e^{(t-s)A}dL_s,\vspace{2mm}\\
\displaystyle Y^{\varepsilon}_t=e^{tA/\varepsilon}y+\frac{1}{\varepsilon}\int^t_0e^{(t-s)A/\varepsilon}F(X^{\varepsilon}_s, Y^{\varepsilon}_s)ds+\frac{1}{\varepsilon^{1/\alpha}}\int^t_0 e^{(t-s)A/\varepsilon}dZ_s.
\end{array}\right.
\end{eqnarray*}
Moreover, for any $T>0$ and $1\leq p<\alpha$, there exists a constant $C_{p,T}>0$ such that
\begin{align}
\sup_{\vare\in(0,1), t\in [0,T]}\mathbb{E}|X_{t}^{\vare}|^{p}\leq C_{p,T}\left(1+|x|^{p}\right),\label{F3.1}
\end{align}
\begin{align}
\sup_{\vare\in(0,1), t\in [0,T]}\mathbb{E}|Y_{t}^{\varepsilon}|^{p}\leq C_{p}\left(1+|y|^{p}\right).\label{ess}
\end{align}
\end{lemma}
\begin{proof}
Put $\mathcal{H}:=H\times H$ and rewrite the system \eref{Main equation} for $V^{\varepsilon}_t=(X^{\varepsilon}_t,Y^{\varepsilon}_t)$ as
\begin{eqnarray*}
dV^{\varepsilon}_t=\tilde{A}V^{\varepsilon}_tdt+B^{\vare}(V^{\varepsilon}_t)dt+d G^{\varepsilon}_t,\quad V^{\varepsilon}_0=(x,y)\in \mathcal{H},
\end{eqnarray*}
where $G^{\varepsilon}_t:=(L_t,\frac{1}{\varepsilon^{1/\alpha}}Z_t)$  is a $\mathcal{H}$-valued cylindrical $\alpha$-stable process, and
\begin{eqnarray*}
\tilde{A}V^{\varepsilon}_t&&\!\!\!\!\!\!\!\!=\left(AX^{\varepsilon}_t,\frac{1}{\varepsilon}AY^{\varepsilon}_t\right),
\\B^{\vare}(V^{\varepsilon}_t)&&\!\!\!\!\!\!\!\!=\left(B(X^{\varepsilon}_t,Y^{\varepsilon}_t),\frac{1}{\varepsilon}F(X^{\varepsilon}_t,Y^{\varepsilon}_t)\right).
\end{eqnarray*}

It is easy to check that $B^{\vare}$ is bounded and H\"{o}lder continuous with index $\eta_1\wedge\eta_2\wedge\eta_3$ in $\mathcal{H}$, i.e.,
$$
|B^{\vare}(v_1)-B^{\vare}(v_2)|_{\mathcal{H}}\leq C_{\vare}|v_1-v_2|^{\eta_1\wedge\eta_2\wedge\eta_3}_{\mathcal{H}},\quad v_1,v_2\in \mathcal{H},
$$
where $|v|_{\mathcal{H}}:=\left(|x|^2_{H}+|y|^2_{H}\right)^{1/2}$ with $v=(x,y)\in \mathcal{H}$. Then under the assumptions \ref{A1}-\ref{A4}, the existence and uniqueness of mild solution for system \eqref{Main equation} follows by \cite[Theorem 2.2]{SXX0} directly.

Next, we intend to prove some a-priori estimates of the solution. Recall that
\begin{eqnarray*}
X^{\varepsilon}_t=e^{tA}x+\int^t_0e^{(t-s)A}B(X^{\varepsilon}_s, Y^{\varepsilon}_s)ds+\int^t_0 e^{(t-s)A}dL_s.
\end{eqnarray*}
By the boundedness of $B$, $\eref{Ls}$ and $\eref{SP1}$, we easily have
\begin{eqnarray*}
\sup_{t\in [0,T]}\mathbb{E}|X^{\varepsilon}_t|^p\leq \!\!\!\!\!\!\!\!&&C_p\sup_{t\in[0,T]}\mathbb{E}\left|e^{tA}x\right|^p+C_p\sup_{t\in[0,T]}\mathbb{E}\left|\int^t_0e^{(t-s)A}d L_s\right|^p\nonumber\\
&&+C_p\sup_{t\in[0, T]}\mathbb{E}\left|\int^t_0e^{(t-s)A}B(X^{\varepsilon}_s,Y^{\varepsilon}_s)ds\right|^p\nonumber\\
\leq\!\!\!\!\!\!\!\!&&C_{p,T}(1+\left|x\right|^p).
\end{eqnarray*}

Recall that
\begin{eqnarray*}
Y^{\varepsilon}_t=e^{tA/\varepsilon}y+\frac{1}{\varepsilon}\int^t_0e^{(t-s)A/\varepsilon}F(X^{\varepsilon}_s, Y^{\varepsilon}_s)ds+\frac{1}{\varepsilon^{1/\alpha}}\int^t_0 e^{(t-s)A/\varepsilon}dZ_s.
\end{eqnarray*}
By \eref{Zs}, we obtain for any $t\geq 0$,
\begin{eqnarray}
\EE\left|\frac{1}{\vare^{1/\alpha}}\int^t_0 e^{(t-s)A/\varepsilon}dZ_s\right|=\EE\left|\int^{t/\vare}_0 e^{(t/\vare-s)A}d\tilde Z_s\right|\leq C\left(\sum_{n\geq 1}\frac{\gamma^{\alpha}_n }{\alpha \lambda_n} \right)^{1/\alpha}.\label{F3.3}
\end{eqnarray}
where $\tilde Z_t:=\frac{1}{\vare^{1/ \alpha}}Z_{t\vare}$, which is also a cylindrical $\alpha$-stable process. Thus by the boundedness of $F$, \eref{SP1} and \eref{F3.3}, we have for any $t\geq 0$
\begin{eqnarray*}
\EE|Y_{t}^{\varepsilon}|^{p}\leq\!\!\!\!\!\!\!\!&&C_p\left[e^{-t\lambda_1 p/\vare }|y|^{p}+\left(\int^t_0 \frac{1}{\vare}e^{-s \lambda_1/\vare}|F(X^{\varepsilon}_s, Y^{\varepsilon}_s)|ds\right)^{p}+\EE\left|\frac{1}{{\vare}^{1/{\alpha}}}\int^t_0 e^{(t-s)A/\vare}d Z_s\right|^{p}\right]\\
\leq\!\!\!\!\!\!\!\!&&C_{p}(1+|y|^{p}).
\end{eqnarray*}
The proof is complete.
\end{proof}

\begin{lemma} \label{SOX}
For any $t\in (0, T]$, $1\leq p<\alpha$, and $\theta\in (0,2/\alpha)$. Then exists a constant $C_{p,T}>0$ such that
\begin{align*}
\sup_{\varepsilon \in (0,1)} \mathbb{E}\|X_{t}^{\varepsilon}\|_{\theta}^{p}
\leq C_{p,T}(t^{-\frac{\theta p}{2}}|x|^{p}+1).
\end{align*}
\end{lemma}
\begin{proof}
By \eref{Ls}, for any $1\leq p<\alpha$ and $\theta\in (0,2/\alpha)$ we have
$$
\sup_{t\in [0,T]}\EE\left\|\int^t_0 e^{(t-s)A}d L_s\right\|^p_{\theta}\leq C_{p}\Big(\sum_{n\geq1}\frac{\beta^{\alpha}_n}{\lambda^{1-\alpha\theta/2}_n}\Big)^{p/\alpha}\leq C_{p}\Big(\sum_{n\geq1}\beta^{\alpha}_n\Big)^{p/\alpha}
$$
which combines with the boundedness of $B$ and \eref{SP2}, it follows
\begin{eqnarray*}
\EE\|X^{\varepsilon}_t\|^p_{\theta}\leq\!\!\!\!\!\!\!\!&&C_p\|e^{tA}x\|^p_{\theta}+C_p\EE\left\|\int^t_0e^{(t-s)A}B(X^{\varepsilon}_s, Y^{\varepsilon}_s)ds\right\|^p_{\theta}+C_p\EE\left\|\int^t_0 e^{(t-s)A}d L_s\right\|^p_{\theta}\\
\leq\!\!\!\!\!\!\!\!&&C_p t^{-\frac{\theta p}{2}}|x|^{p}+C_p\Big[\int^t_0 (t-s)^{-\frac{\theta}{2}}ds\Big]^{p}+C_{p}\\
\leq\!\!\!\!\!\!\!\!&&C_{p,T}(t^{-\frac{\theta p}{2}}|x|^{p}+1).
\end{eqnarray*}
The proof is complete.
\end{proof}

\begin{lemma} \label{COX}
For any $T>0$, $1\leq p<\alpha$ and $\theta\in (0,2/\alpha)$.Then there exists a constant $C_{p,T}>0$ such that for any $\vare\in(0,1)$ and $\delta>0$ small enough,
\begin{align}
\mathbb{E}\left[\int^{T}_0|X^{\varepsilon}_t-X^{\varepsilon}_{t(\delta)}|dt\right]^{p}\leq C_{p,T}(1+|x|^p)\delta^{\frac{p\theta}{2}},\label{F3.7}
\end{align}
where $t(\delta):=[\frac{t}{\delta}]\delta$ with $[s]$ denotes the integer part.
\end{lemma}

\begin{proof}
For any $T>0$ and $1\leq p<\alpha$, by \eref{F3.1} it follows
\begin{eqnarray}
\mathbb{E}\left[\int^{T}_0|X_{t}^{\varepsilon}-X_{t(\delta)}^{\varepsilon}|dt\right]^{p}\leq\!\!\!\!\!\!\!\!&&C_p\mathbb{E}\left[\int^{\delta}_0|X_{t}^{\varepsilon}-x|dt\right]^{p}+C_p\mathbb{E}\left[\int^{T}_{\delta}|X_{t}^{\varepsilon}-X_{t(\delta)}^{\varepsilon}|dt\right]^{p}\nonumber\\
\leq\!\!\!\!\!\!\!\!&&C_{p,T}\delta^{p}(1+|x|^p)+C_{p}\mathbb{E}\left(\int^{T}_{\delta}|X_{t}^{\varepsilon}-X_{t-\delta}^{\varepsilon}|dt\right)^{p}\nonumber\\
&&+C_{p}\mathbb{E}\left(\int^{T}_{\delta}|X_{t(\delta)}^{\varepsilon}-X_{t-\delta}^{\varepsilon}|dt\right)^{p}.\label{FFX1}
\end{eqnarray}
It is easy to see that
\begin{eqnarray}
X_{t}^{\varepsilon}-X_{t-\delta}^{\varepsilon}=\!\!\!\!\!\!\!\!&&(e^{A\delta}-I)X_{t-\delta}^{\varepsilon}+\int_{t-\delta}^{t}e^{(t-s)A}B(X^{\varepsilon}_s, Y^{\varepsilon}_s)ds+\int_{t-\delta}^{t}e^{(t-s)A}dL_s  \nonumber\\
:=\!\!\!\!\!\!\!\!&&I_{1}(t)+I_{2}(t)+I_{3}(t). \label{REGX}
\end{eqnarray}

By Minkowski's inequality and Lemma \ref{SOX}, for any $1\leq p<\alpha$ and $\theta\in (0,2/\alpha)$ we have
\begin{eqnarray}  \label{REGX1}
\mathbb{E}\left(\int^{T}_{\delta}|I_{1}(t)| dt\right)^p
\leq\!\!\!\!\!\!\!\!&&C_p\delta^{\frac{\theta p}{2}}\left[\int^{T}_\delta\left(\mathbb{E} \|X^{\varepsilon}_{t-\delta}\|^p_{\theta}  \right)^{1/p}dt\right]^{p}\nonumber\\
\leq\!\!\!\!\!\!\!\!&&C_{p,T}\delta^{\frac{\theta p}{2}}(|x|^p+1).
\end{eqnarray}

The boundedness of $B$ imply that
\begin{eqnarray} \label{REGX3}
\mathbb{E}\left[\int^{T}_\delta\!\!\left|\int_{t-\delta}^{t}e^{(t-s)A}B(X^{\varepsilon}_s,Y^{\varepsilon}_s)ds\right |dt\right]^{p}
\!\!\!\leq\!\!\!\!\!\!\!\!&&C_{p,T}\delta^{p}
\end{eqnarray}

By Minkowski's inequality and \cite[(4.12)]{PZ}, we obtain
\begin{eqnarray} \label{REGX4}
\mathbb{E}\left[\int^{T}_\delta \left|\int_{t-\delta}^{t}e^{(t-s)A}dL_s \right|dt\right]^{p}\leq\!\!\!\!\!\!\!\!&&\left\{\int^{T}_\delta \left[\mathbb{E}\left|\int_{t-\delta}^{t}e^{(t-s)A}dL_s \right|^p\right]^{1/p}dt\right\}^{p}\nonumber\\
\leq\!\!\!\!\!\!\!\!&&\left\{\int^{T}_\delta\left[ \sum^{\infty}_{k=1}\frac{(1-e^{-\alpha\lambda_k \delta})\beta^{\alpha}_k}{\lambda_k}\right]^{1/\alpha}dt\right\}^{p}\nonumber\\
\leq\!\!\!\!\!\!\!\!&&\left\{\int^{T}_\delta\left[ \sum^{\infty}_{k=1}\frac{\beta^{\alpha}_k}{\lambda^{1-\alpha\theta/2}_k}\right]^{1/\alpha}dt\right\}^{p}\nonumber\\
\leq\!\!\!\!\!\!\!\!&&C_{p,T}\delta^{\frac{\theta p}{2}}.
\end{eqnarray}
Combining \eqref{REGX1}-\eqref{REGX4}, we obtain
\begin{eqnarray}\label{FFX2}
\mathbb{E}\left(\int^{T}_{\delta}|X_{t}^{\varepsilon}-X_{t-\delta}^{\varepsilon}|dt\right)^{p}\leq\!\!\!\!\!\!\!\!&&C_{p,T}\delta^{\frac{p\theta}{2}}(1+|x|^p).
\end{eqnarray}
Similar as the argument above, we also have
\begin{eqnarray}\label{FFX3}
\mathbb{E}\left(\int^{T}_{\delta}|X_{t(\delta)}^{\varepsilon}-X_{t-\delta}^{\varepsilon}|dt\right)^{p}\leq\!\!\!\!\!\!\!\!&&C_{p,T}\delta^{\frac{p\theta}{2}}(1+|x|^p).
\end{eqnarray}

Finally, \eref{FFX1}, (\ref{FFX2}) and (\ref{FFX3}) imply Lemma \ref{COX} holds. The proof is complete.
\end{proof}

\vskip 0.3cm

\subsection{The frozen and averaged equations}

For fixed $x\in H$, we consider the corresponding frozen equation:
\begin{eqnarray}\label{FEQ1}
dY_{t}=[AY_{t}+F(x,Y_{t})]dt+d Z_{t},\quad Y_{0}=y.
\end{eqnarray}
Since $F(x,\cdot)$ is Lipshcitz continuous, it is easy to check that the equation $\eref{FEQ1}$ has a unique mild solution $\{Y_{t}^{x,y}\}_{t\geq 0}$ for any initial value $y\in H$. Following a similar argument as in the proof of \eref{ess}, we can easily obtain
\begin{eqnarray*}
\sup_{t\geq 0}\EE|Y_{t}^{x,y}|^p\leq C_p(1+|y|^p).
\end{eqnarray*}

Let $P^{x}_t$ be the transition semigroup of $Y_{t}^{x,y}$, i.e., for any bounded measurable function $\varphi$ on $H$,
\begin{eqnarray*}
P^x_t \varphi(y)= \mathbb{E} \left[\varphi\left(Y_{t}^{x,y}\right)\right], \quad y \in H, t>0.
\end{eqnarray*}
Refer to \cite[Lemma 3.3]{BYY}, $\{P^x_t\}_{t\geq 0}$ admits a unique invariant measure $\mu^x$, which satisfies
$$
\sup_{x\in H}\int_{H}|z|^{p}\mu^{x}(dz)<\infty, 0<p<\alpha.
$$

Before proving the asymptotical behavior of $P^x_t$, we first give the following Lemma.

\begin{lemma} \label{L3.4} There exists a constant $C>0$ such that for any $x_1,x_2, y_1,y_2\in H$ and $t\geq 0$,
\begin{eqnarray*}
|Y^{x_1,y_1}_t-Y^{x_2,y_2}_t|^2\leq C|x_1-x_2|^{2\eta_3}+e^{-(\lambda_1-L_F) t}|y_1-y_2|^{2}.
\end{eqnarray*}
\end{lemma}
\begin{proof}
Note that for any $x_1,x_2, y_1,y_2\in H$ and $t\geq 0$,
\begin{eqnarray*}
d(Y^{x_1,y_1}_t-Y^{x_2,y_2}_t)=\!\!\!\!\!\!\!\!&&A(Y^{x_1,y_1}_t-Y^{x_2,y_2}_t)dt+\left[F(x_1, Y^{x_1,y_1}_t)-F(x_2, Y^{x_2,y_2}_t)\right]dt.
\end{eqnarray*}
By $\lambda_1-L_F>0$ in assumption \ref{A4} and Young's inequality, we have
\begin{eqnarray*}
&&\frac{d}{dt}|Y^{x_1,y_1}_t-Y^{x_2,y_2}_t|^2\\
=\!\!\!\!\!\!\!\!&&-2\|Y^{x_1,y_1}_t-Y^{x_2,y_2}_t\|_1^2+2\langle F(x_1, Y^{x_1,y_1}_t)-F(x_2, Y^{x_2,y_2}_t), Y^{x_1,y_1}_t-Y^{x_2,y_2}_t\rangle\\
\leq\!\!\!\!\!\!\!\!&&-2\lambda_1|Y^{x_1,y_1}_t-Y^{x_2,y_2}_t|^2+2L_F|Y^{x_1,y_1}_t-Y^{x_2,y_2}_t|^2+C|x_1-x_2|^{\eta_3}|Y^{x_1,y_1}_t-Y^{x_2,y_2}_t|\\
\leq\!\!\!\!\!\!\!\!&&-(\lambda_1-L_F)|Y^{x_1,y_1}_t-Y^{x_2,y_2}_t|^2+C|x_1-x_2|^{2\eta_3}.
\end{eqnarray*}
Using the comparison theorem, it follows
\begin{eqnarray*}
|Y^{x_1,y_1}_t-Y^{x_2,y_2}_t|^2\leq\!\!\!\!\!\!\!\!&&e^{-(\lambda_1-L_F)t}|y_1-y_2|^2+C\int^t_0 e^{-(\lambda_1-L_F)(t-s)}ds |x_1-x_2|^{2\eta_3}\\
\leq\!\!\!\!\!\!\!\!&&C|x_1-x_2|^{2\eta_3}+e^{-(\lambda_1-L_F)t}|y_1-y_2|^{2}.
\end{eqnarray*}
The proof is complete.
\end{proof}

Now, we give a position to prove the following exponential behavior of  transition semigroup $\{P^x_t\}_{t\geq 0}$.
\begin{proposition}\label{Ergodicity}
There exists $C>0$  such that for any H\"{o}lder continuous function $\varphi: H\rightarrow H$ with index $\beta\in (0,1)$ and $x,y\in H$,
\begin{equation*}
\Big|P^x_t\varphi(y)-\int_{H}\varphi(z)\mu^x(dz)\Big|\leq C(1+|y|^{\beta})\|\varphi\|_{C^{\beta}}e^{-\frac{(\lambda_1-L_F)\beta t}{2}}.
\end{equation*}
\end{proposition}
\begin{proof}
For any given H\"{o}lder continuous function $\varphi:H\rightarrow H$ with index $\beta\in (0,1)$. By the definition of invariant measure and Lemma \ref{L3.4}, we have
\begin{eqnarray*}
\left|P^{x}_t\varphi(y)-\int_{H}\varphi(z)\mu^{x}(dz)\right|\leq\!\!\!\!\!\!\!\!&&\int_{H}\left|\mathbb{E}\varphi(Y^{x,y}_t)-\EE\varphi(Y^{x,z}_t)\right|\mu^{x}(dz)\\
\leq\!\!\!\!\!\!\!\!&&\|\varphi\|_{C^{{\beta}}}\int_{H}\mathbb{E}\left|Y^{x,y}_t-Y^{x,z}_t\right|^{\beta}\mu^{x}(dz)\\
\leq\!\!\!\!\!\!\!\!&&\|\varphi\|_{C^{{\beta}}}\int_{H}e^{-\frac{(\lambda_1-L_F)\beta t}{2} }|y-z|^{\beta}\mu^{x}(dz)\\
\leq\!\!\!\!\!\!\!\!&&C(1+|y|^{\beta})\|\varphi\|_{C^{{\beta}}}e^{-\frac{(\lambda_1-L_F)\beta t}{2} },
\end{eqnarray*}
where the last inequality comes from $\int_{H}|z|^{\beta}\mu^{x}(dz)<\infty$. The proof is complete.
\end{proof}

Next, we consider the averaged equation, i.e.,
\begin{equation}
\displaystyle d\bar{X}_{t}=A\bar{X}_{t}dt+\bar{B}(\bar{X}_{t})dt+dL_t,\quad
\bar{X}_{0}=x, \label{3.1}
\end{equation}
where $\bar{B}(x):=\int_{H}B(x,y)\mu^{x}(dy)$ and $\mu^{x}$ is the unique invariant measure of the transition semigroup of  equation $\eref{FEQ1}$.

\vskip 0.1cm

\begin{lemma}\label{L3.8}
For any $x\in H$, the equation \eref{3.1} has a unique mild solution $\bar{X}$.
Moreover, for any $T>0$ and $1\leq p<\alpha$, there exists a constant $C_{p,T}>0$ such that
\begin{align}
\sup_{t\in[0,T]}\mathbb{E}|\bar{X}_{t}|^{p}\leq C_{p,T}(1+|x|^{p}).\label{EbarX}
\end{align}
\end{lemma}
\begin{proof}
It is sufficient to show that the coefficient $\bar{B}$ is bounded and H\"{o}lder continuous with index $\eta_1\wedge(\eta_2\eta_3)$, i.e.,
\begin{eqnarray}
|\bar{B}(x_1)-\bar{B}(x_2)|\leq C|x_1-x_2|^{\eta_1\wedge(\eta_2\eta_3)},\quad  x_1,x_2\in H.\label{HObarB}
\end{eqnarray}
Then the existence and uniqueness of the mild solution and estimate \eref{EbarX} can be proved by a similar argument as in the proof of Lemma \ref{PMY}.

In fact, the boundedness of $\bar{B}$ holds obviously due to the boundedness of $B$. By Proposition \ref{Ergodicity} and Lemma \ref{L3.4}, we obtain
\begin{eqnarray*}
|\bar{B}(x_1)-\bar{B}(x_2)|=\!\!\!\!\!\!\!\!&&\left|\int_{H} B(x_1,z)\mu^{x_1}(dz)-\int_{H} B(x_2,z)\mu^{x_2}(dz)\right|\\
\leq\!\!\!\!\!\!\!\!&&\left|\int_{H} B(x_1,z)\mu^{x_1}(dz)-{\EE} B(x_1, Y^{x_1,0}_t)\right|\\
&&+\left|  \EE B(x_2, Y^{x_2,0}_t)-\int_{H} B(x_2,z)\mu^{x_2}(dz)\right|\\
&&+ {\EE} \left|B(x_1, Y^{x_1,0}_t)-B(x_2, Y^{x_1,0}_t)\right|+ {\EE} \left|B(x_2, Y^{x_1,0}_t)-B(x_2, Y^{x_2,0}_t)\right|\\
\leq\!\!\!\!\!\!\!\!&&C e^{-\frac{(\lambda_1-L_F)\eta_2 t}{2}}+C{\EE} \left|B(x_1, Y^{x_1,0}_t)-B(x_2, Y^{x_1,0}_t)\right|^{\frac{\eta_1\wedge(\eta_2\eta_3)}{\eta_1}}\\
&&+ C{\EE} \left|B(x_2, Y^{x_1,0}_t)-B(x_2, Y^{x_2,0}_t)\right|^{\frac{\eta_1\wedge(\eta_2\eta_3)}{\eta_2\eta_3}}\\
\leq\!\!\!\!\!\!\!\!&&C e^{-\frac{(\lambda_1-L_F)\eta_2 t}{2}}+C|x_1-x_2|^{\eta_1\wedge(\eta_2\eta_3)},
\end{eqnarray*}
where the second inequality comes from the boundedness of $B$. Hence, letting $t\rightarrow \infty$, we obtain (\ref{HObarB}). The proof is complete.
\end{proof}

\vskip 0.3cm

\section{Main result and its proof}

In this section, we devote to proving our main result. The main techniques are based on Zvonkin's transformation (see subsection 4.1) and Khasminkii's time discretization (see subsection 4.2).
The former is  now widely used to study the strong well-posedness for S(P)DEs with singular coefficients (see e.g. \cite{DF,DFPR,KR}), and the latter is the classical method used to study the
averaging principle for kinds of slow-fast S(P)DEs (see e.g. \cite{C1,K1}).

\subsection{The Zvonkin's transformation}

Since the coefficients of the system \eref{Main equation} are only H\"older continuous, the classical Khasminskii's time discreatization can't be used to prove our main result directly.
Inspired from \cite{V0}, the key technique here is using the Zvonkin's transformation, i.e., changing the singular coefficients to regular ones.

Now, considering the following partial differential equation in $H$:
\begin{equation}\label{MPDE}
\lambda U_{\lambda}(x)-\bar{\mathscr{L}}U_{\lambda}(x)=\bar{B}(x),\quad x\in H,
\end{equation}
where $\lambda>0$ and $\bar{\mathscr{L}}$ is the infinitesimal generator of the averaged equation \eref{3.1}, i.e.,
\begin{eqnarray}\label{gege}
\bar{\mathscr{L}}f(x):=\!\!\!\!\!\!\!\!&&\langle Ax, Df(x)\rangle+\langle \bar{B}(x), Df(x)\rangle\nonumber\\
&&\!\!+\sum_{k\geq 1}\beta^{\alpha}_k\int_\mathbb{R}\left[f(x+{e_k}z)-f(x)-\langle Df(x),{e_k}z\rangle 1_{|z|\leq c}\right]\frac{C_\alpha}{|z|^{1+\alpha}}dz.
\end{eqnarray}

We first state the following key Lemma.
\begin{lemma}\label{ppp}
For $\lambda>0$ and $ \theta'\in(0,\eta_1\wedge(\eta_2\eta_3))$, there exists a function $U_{\lambda}\in C^{\gamma+\theta'}_b(H,H)$ satisfying the following integral equation:
\begin{align}
U_{\lambda}(x)=\int_0^\infty\!\e^{-\lambda t}T_t\Big(\langle \bar{B}, DU_{\lambda}\rangle+\bar{B}\Big)(x)dt,   \label{inte}
\end{align}
where $T_t f(x):=\EE[f(M^{x}_t)]$ for any $f\in B_b(H,H)$ (the set of all bounded and Borel measurable function from $H$ to $H$), $M_t^x$ is the unique solution of following equation:
$$dM_t^x=AM_t^xdt+dL_t,\quad M_0^x=x\in H.$$
Moreover, $U_{\lambda}$ also solves equation \eref{MPDE} and the following estimates hold:
\begin{eqnarray}
\|U_{\lambda}\|_{\gamma+\theta^{'}}\leq \tilde{C}_{\lambda}\|\bar{B}\|_{C^{\eta_1\wedge(\eta_2\eta_3)}_b},\quad \forall \theta'\in(0,\eta_1\wedge(\eta_2\eta_3));\label{BU}\end{eqnarray}
\begin{eqnarray}\|(-A)^{\kappa_1}DU_{\lambda}\|_{\infty}\leq C\|\bar{B}\|_{C^{\eta_1\wedge(\eta_2\eta_3)}_b},\label{BU2}
\end{eqnarray}
where $\tilde{C}_{\lambda}>0$ is a constant which satisfies $\lim_{\lambda\to \infty}\tilde{C}_{\lambda}=0$.
\end{lemma}
\begin{proof}
The existence of solution of equation \eref{inte}, its solution solves equation \eref{MPDE} and the estimate \eref{BU} have been proved follows a standard argument (see e.g. \cite[Theorem 3.3]{SXX0}). Thus we only prove $\eref{BU2}$ here.

By following a similar argument as in the proof of \cite[(3.3)]{SXX0}, it is easy to prove that for any $\kappa_1\in(0,1/2)$,
$$\|(-A)^{\kappa_1}D T_t f\|_{\infty}\leq C\Lambda_{3,\kappa_1}(t)\|f\|_{\infty},$$
where  $\Lambda_{3,\kappa_1}(t) := \sup\limits_{n\geq1}\frac{e^{-\lambda_nt}\lambda^{\kappa_1+1/\alpha}_n}{\beta_n}$. Thus by assumption \eref{A31}, we final get
\begin{align*}
\|(-A)^{\kappa_1}DU_{\lambda}\|_{\infty}&= \int_0^\infty\!\e^{-\lambda t}\left\|(-A)^{\kappa_1}DT_t\Big(\langle \bar{B}, DU_{\lambda}\rangle+\bar{B}\Big)\right\|_{\infty}dt\\
&\leq C\int_0^\infty\!\e^{-\lambda t}\Lambda_{3,\kappa_1}(t)dt\cdot\|\langle \bar{B}, DU_{\lambda}\rangle+\bar{B}\|_{\infty}\\
&\leq C\|\bar{B}\|_{C^{\eta_1\wedge(\eta_2\eta_3)}_b}.
\end{align*}
The whole proof is finished.
\end{proof}

Now, we have the following Zvonkin's transformation.
\begin{lemma}
For any given $\lambda>0$, let $U_{\lambda}$ be the solution of equation \eref{MPDE}. Then the solution $\bar X_t$ of the averaged equation \eref{3.1} satisfies
\begin{eqnarray}
\bar X_t=\!\!\!\!\!\!\!\!&&e^{tA}(x+U_{\lambda}(x))+\int^t_0 e^{(t-s)A}\lambda U_{\lambda}(\bar X_s)ds-U_{\lambda}(\bar X_t)-\int^t_0 Ae^{(t-s)A}U_{\lambda}(\bar X_s)ds\nonumber\\
&&\!\!\!\!+\int^t_0 e^{(t-s)A}d L_s+\sum_{k\geq 1}\int^t_0 \int_\mathbb{R}e^{(t-s)A}\left[U_{\lambda}(\bar{X}_{s^-}+\beta_k e_k z)-U_{\lambda}(\bar{X}_{s^-})\right]\tilde{N}^{1,k}(ds,dz)\label{FbarX}
\end{eqnarray}
and
\begin{eqnarray}
X^{\vare}_t=\!\!\!\!\!\!\!\!&&e^{tA}(x+U_{\lambda}(x))+\int^t_0 e^{(t-s)A}\lambda U_{\lambda}(X^{\vare}_s)ds-U_{\lambda}(X^{\vare}_t)-\int^t_0 Ae^{(t-s)A}U_{\lambda}(X^{\vare}_s)ds\nonumber\\
&&\!\!+\int^t_0 e^{(t-s)A}d L_s+\sum_{k\geq 1}\int^t_0 \int_\mathbb{R}e^{(t-s)A}\left[U_{\lambda}(X^{\varepsilon}_{s^-}+\beta_k e_k z)-U_{\lambda}(X^{\varepsilon}_{s^-})\right]\tilde{N}^{1,k}(ds,dz)\nonumber\\
&&\!\!+\int^t_0 e^{(t-s)A}\langle I+DU_{\lambda}(X^{\vare}_s), B(X^{\vare}_s, Y^{\vare}_s)-\bar{B}(X^{\vare}_s)\rangle ds,\label{FX}
\end{eqnarray}
where $I$ is the identical operator.
\end{lemma}
\begin{proof}
By Lemma \ref{ppp}, we have $U_{\lambda}\in C^{\gamma+\theta'}_b(H,H)$ for any $\theta'\in(0, \eta_1\wedge(\eta_2\eta_3))$. Then by It\^o's formula, we have
\begin{eqnarray*}
dU_{\lambda}(\bar{X}_t)=\!\!\!\!\!\!\!\!&&\bar{\mathscr{L}} U_{\lambda}(\bar{X}_t)dt+\sum_{k\geq1}\int_{\RR}\left[U_{\lambda}(\bar{X}_{t-}+\beta_k e_k z)-U_{\lambda}(\bar{X}_{t-})\right]\tilde{N}^{1,k}(dt, dz)\\\\
=\!\!\!\!\!\!\!\!&&\lambda U_{\lambda}(\bar X_t)dt-\bar B(\bar X_t)dt+\sum_{k\geq 1}\int_\mathbb{R}[U_{\lambda}(\bar{X}_{t^-}+\beta_k e_k z)-U_{\lambda}(\bar{X}_{t^-})]\tilde{N}^{1,k}(dt,dz).
\end{eqnarray*}
Note that the $\alpha$-stable process is considered in our case,  we can apply It\^o's formula on $\bar{X}_t$ for any $\varphi\in C^{r}_b(H, H )$ for any $r>\alpha$ (see \cite[Lemma 4.1]{SXX0}). So we can take $\theta'$ close to $\eta_1\wedge(\eta_2\eta_3)$ such that $\gamma+\theta'>\alpha$. As a result, it follows
\begin{eqnarray*}
\bar B(\bar X_t)dt=\lambda U_{\lambda}(\bar X_t)dt-dU_{\lambda}(\bar{X}_t)+\sum_{k\geq 1}\int_\mathbb{R}[U_{\lambda}(\bar{X}_{t^-}+\beta_k e_k z)-U_{\lambda}(\bar{X}_{t^-})]\tilde{N}^{1,k}(dt,dz).
\end{eqnarray*}
Substituting this formula into equation (\ref{3.1}), it follows that
\begin{eqnarray*}
d\bar X_t=\!\!\!\!\!\!\!\!&&A\bar X_t dt+\lambda U_{\lambda}(\bar X_t)dt-dU_{\lambda}(\bar X_t)\\
&&+\sum_{k\geq 1}\int_\mathbb{R}\left[U_{\lambda}(\bar{X}_{t^-}+\beta_k e_k z)-U_{\lambda}(\bar{X}_{t^-})\right]\tilde{N}^{1,k}(dt,dz)+d L_t.
\end{eqnarray*}
By variation of constant method and integration by parts formula, we get
\begin{eqnarray*}
\bar X_t=\!\!\!\!\!\!\!\!&&e^{tA}(x+U_{\lambda}(x))+\int^t_0 e^{(t-s)A}\lambda U_{\lambda}(\bar X_s)ds-U_{\lambda}(\bar X_t)-\int^t_0 Ae^{(t-s)A}U_{\lambda}(\bar X_s)ds\nonumber\\
&&+\int^t_0 e^{(t-s)A}d L_s+\sum_{k\geq 1}\int^t_0 \int_\mathbb{R}e^{(t-s)A}\left[U_{\lambda}(\bar{X}_{s^-}+\beta_k e_k z)-U_{\lambda}(\bar{X}_{s^-})\right]\tilde{N}^{1,k}(ds,dz).
\end{eqnarray*}

Note that we can rewrite $X^{\varepsilon}_t$ as follows:
$$
dX^{\varepsilon}_t=A X^{\varepsilon}_tdt+\bar{B}(X^{\varepsilon}_t)dt+\left[B(X^{\varepsilon}_t, Y^{\varepsilon}_t)-\bar{B}(X^{\varepsilon}_t)\right]dt+d L_{t}.
$$
As the proof of \eref{FbarX}, we have
\begin{eqnarray*}
X^{\vare}_t=\!\!\!\!\!\!\!\!&&e^{tA}(x+U_{\lambda}(x))+\int^t_0 e^{(t-s)A}\lambda U_{\lambda}(X^{\vare}_s)ds-U_{\lambda}(X^{\vare}_t)-\int^t_0 Ae^{(t-s)A}U_{\lambda}(X^{\vare}_s)ds\nonumber\\
&&\!\!+\int^t_0 e^{(t-s)A}d L_s+\sum_{k\geq 1}\int^t_0 \int_\mathbb{R}e^{(t-s)A}\left[U_{\lambda}(X^{\varepsilon}_{s^-}+\beta_k e_k z)-U_{\lambda}(X^{\varepsilon}_{s^-})\right]\tilde{N}^{1,k}(ds,dz)\nonumber\\
&&\!\!+\int^t_0 e^{(t-s)A}\langle I+D U_{\lambda}(X^{\vare}_s), B(X^{\vare}_s, Y^{\vare}_s)-\bar{B}(X^{\vare}_s)\rangle ds.
\end{eqnarray*}
The proof is complete.
\end{proof}
\begin{remark}
Note that the non-regular drift $B$ has been removed in \eref{FbarX} and several new terms appear, which will be proved to be Lipschitz continuous. Although that the last term in \eref{FX} is still non-regular, it is possible to be handled by Khasminskii's time discreatization and the exponential ergodicity of the transition semigroup of the frozen equation.
\end{remark}

\vskip 0.3cm

\subsection{The Khasminkii's time discretization}

Following the idea in \cite{K1}, we introduce an auxiliary process $(\hat{X}_{t}^{\varepsilon},\hat{Y}_{t}^{\varepsilon})\in{H}\times H$ and divide $[0,T]$ into intervals of size $\delta$, where $\delta$ is a fixed positive number depending on $\vare$ and will be chosen later.

We construct a process $\hat{Y}_{t}^{\varepsilon}$, with $\hat{Y}_{0}^{\varepsilon}=Y^{\varepsilon}_{0}=y$, and for any $k\in \mathbb{N}$ and $t\in[k\delta,(k+1)\delta\wedge T]$,
\begin{eqnarray*}
\hat{Y}_{t}^{\varepsilon}=\hat Y_{k\delta}^{\varepsilon}+\frac{1}{\varepsilon}\int_{k\delta}^{t}A\hat{Y}_{s}^{\varepsilon}ds+\frac{1}{\varepsilon}\int_{k\delta}^{t}
F(X_{k\delta}^{\varepsilon},\hat{Y}_{s}^{\varepsilon})ds+\frac{1}{\varepsilon^{1/\alpha}}\int_{k\delta}^{t}dZ_s,
\end{eqnarray*}
which satisfies
$$
\hat{Y}_{t}^{\vare}=e^{tA/\vare}y+\frac{1}{\vare}\int_0^te^{(t-s)A/\vare}F\left(X^{\vare}_{s(\delta)},\hat{Y}_{s}^{\vare}\right)ds+\frac{1}{\varepsilon^{1/\alpha}}\int_0^te^{(t-s)A/\vare}dZ_s\quad t\in [0,T].
$$
We also construct another auxiliary process $\hat{X}_{t}^{\vare}\in H$ by the following way:
\begin{eqnarray}
\hat X^{\vare}_t:=\!\!\!\!\!\!\!\!&&e^{tA}(x+U_{\lambda}(x))+\int^t_0 e^{(t-s)A}\lambda U_{\lambda}(X^{\vare}_s)ds-U_{\lambda}(X^{\vare}_t)-\int^t_0 Ae^{(t-s)A}U_{\lambda}(X^{\vare}_s)ds\nonumber\\
&&\!\!+\int^t_0 e^{(t-s)A}dL_s+\sum_{k\geq 1}\int^t_0\int_\mathbb{R} e^{(t-s)A}[U_{\lambda}(X^\varepsilon_{s^{-}}+\beta_k z e_k)-U_{\lambda}(X^\varepsilon_{s^{-}})]\tilde{N}^{1,k}(ds,dz)\nonumber\\
&&\!\!+\int^t_0 e^{(t-s(\delta))A}\langle D U_{\lambda}(X^{\vare}_{s(\delta)})+I, B(X^{\vare}_{s(\delta)}, \hat Y^{\vare}_s)-\bar{B}(X^{\vare}_{s(\delta)})\rangle ds.\label{AMX}
\end{eqnarray}

\begin{lemma} \label{DEY}
For any $T>0$ and $\theta\in (0,2/\alpha)$, there exists a constant $C_{T}>0$ such that
\begin{eqnarray*}
\mathbb{E}\left(\int_0^{T}|Y_{t}^{\varepsilon}-\hat{Y}_{t}^{\varepsilon}|dt\right)\leq C_{T}(1+|x|)\delta^{\frac{\theta\eta_3}{2}}.
\end{eqnarray*}
\end{lemma}

\begin{proof}
By the construction of $Y_t^{\vare}$ and $\hat{Y}_t^{\vare}$, it is easy to see that
\begin{eqnarray*}
Y_t^{\vare}-\hat{Y}_t^{\vare}=\frac{1}{\vare}\int_0^t e^{(t-s)A/\vare}\left[F\left(X^{\vare}_s,Y_{s}^{\vare}\right)-F\left(X^{\vare}_{s(\delta)},\hat{Y}_{s}^{\vare}\right)\right]ds,
\end{eqnarray*}
which implies that for any $t\in [0,T]$,
\begin{eqnarray*}
|Y_t^{\vare}-\hat{Y}_t^{\vare}|\leq \frac{1}{\vare}\int_0^t e^{-\lambda_1(t-s)/\vare}\left[C|X^{\vare}_s-X^{\vare}_{s(\delta)}|^{\eta_3}+L_F|Y_{s}^{\vare}-\hat{Y}_{s}^{\vare}|\right]ds.
\end{eqnarray*}
By Fubini's theorem, we get
\begin{eqnarray*}
\int_0^T|Y_t^{\vare}-\hat{Y}_t^{\vare}|dt\leq\!\!\!\!\!\!\!\!&& \frac{1}{\vare}\int_0^T\int_0^t e^{-\lambda_1(t-s)/\vare}\left[C|X^{\vare}_s-X^{\vare}_{s(\delta)}|^{\eta_3}+L_F|Y_{s}^{\vare}-\hat{Y}_{s}^{\vare}|\right]dsdt\\
=\!\!\!\!\!\!\!\!&& \frac{1}{\vare}\int_0^T\int_s^T e^{-\lambda_1(t-s)/\vare}\left[C|X^{\vare}_s-X^{\vare}_{s(\delta)}|^{\eta_3}+L_F|Y_{s}^{\vare}-\hat{Y}_{s}^{\vare}|\right]dtds\\
\leq\!\!\!\!\!\!\!\!&& \frac{1}{\lambda_1}\int_0^T\left[C|X^{\vare}_s-X^{\vare}_{s(\delta)}|^{\eta_3}+L_F|Y_{s}^{\vare}-\hat{Y}_{s}^{\vare}|\right]ds.
\end{eqnarray*}
Note that $\lambda_1-L_F>0$ in the assumption \ref{A4}, H\"{o}lder inequality and Lemma \ref{COX} yield that
\begin{eqnarray*}
\EE\left(\int_0^T|Y_t^{\vare}-\hat{Y}_t^{\vare}|dt\right)
\leq C\left(\EE\int_0^T|X^{\vare}_t-X^{\vare}_{t(\delta)}|dt\right)^{\eta_3}
\leq C_{T}(1+|x|)\delta^{\frac{\theta\eta_3}{2}}.
\end{eqnarray*}
The proof is complete.
\end{proof}

\begin{lemma} \label{DEX} For any $T>0$, $p\geq 1$, there exists $C_{p,T}>0$ and $\tilde{\theta}>0$ such that
\begin{eqnarray}
\mathbb{E}\Big(\sup_{t\in [0, T]}|X_{t}^{\vare}-\hat{X}_{t}^{\vare}|^p\Big)\leq C_{p,T}(1+|x|)\delta^{\tilde{\theta}}.\label{F4.9}
\end{eqnarray}
\end{lemma}

\begin{proof}
Note that the boundedness and H\"{o}lder continuous of $DU_{\lambda}$, $B$ and $\bar{B}$. Then it follows from \eref{FX} and \eref{AMX} we have
\begin{eqnarray*}
|X^{\vare}_t-\hat{X}^{\vare}_t|\leq\!\!\!\!\!\!\!\!&&\Big|\int^t_0 \!\!e^{(t-s)A}\langle D U_{\lambda}(X^{\vare}_{s}), B(X^{\vare}_{s}, Y^{\vare}_s)-\bar{B}(X^{\vare}_{s})\rangle\\
&&\quad\quad-e^{(t-s(\delta))A}\langle D U_{\lambda}(X^{\vare}_{s(\delta)}), B(X^{\vare}_{s(\delta)}, \hat Y^{\vare}_s)-\bar{B}(X^{\vare}_{s(\delta)})\rangle ds\Big|\nonumber\\
&&\!\!+\left|\int^t_0 e^{(t-s)A}\left[B(X^{\vare}_{s},  Y^{\vare}_s)-\bar{B}(X^{\vare}_{s})\right]-e^{(t-s(\delta))A}\left[B(X^{\vare}_{s(\delta)}, \hat Y^{\vare}_s)-\bar{B}(X^{\vare}_{s(\delta)})\right] ds\right|\\
\leq\!\!\!\!\!\!\!\!&&C\int^t_0 \!\!\|e^{(t-s)A}-e^{(t-s(\delta))A}\|ds+C\int^t_0\|D U_{\lambda}(X^{\vare}_s)-DU_{\lambda}(X^{\vare}_{s(\delta)})\|ds\\
&&+C\int^t_0 \left|B(X^{\vare}_s, Y^{\vare}_s)-\bar{B}(X^{\vare}_s)-B(X^{\vare}_{s(\delta)},\hat Y^{\vare}_s)+\bar{B}(X^{\vare}_{s(\delta)})\right|ds.
\end{eqnarray*}
Using \eref{BU} and the boundedness and H\"{o}lder continuous of $DU_{\lambda}$, $B$ and $\bar{B}$ again,  we get for any $T>0$, $p\geq 1$ and $\theta'\in(0,\eta_1\wedge(\eta_2\eta_3))$,
\begin{eqnarray*}
\EE\left(\sup_{t\in[0, T]}|X^{\vare}_t-\hat{X}^{\vare}_t|^p\right)
\leq\!\!\!\!\!\!\!\!&&C_p\delta^{ p/2}\left(\int^T_0 s^{-1/2}ds\right)^p+C_{p,T}\EE\int^T_0|X^{\vare}_s-X^{\vare}_{s(\delta)}|^{\gamma+\theta'-1} ds\\
&&+C_{p,T}\EE\int^T_0|X^{\vare}_s-X^{\vare}_{s(\delta)}|^{\eta_1}ds+C_{p,T}\EE\int^T_0|X^{\vare}_s-X^{\vare}_{s(\delta)}|^{\eta_1\wedge(\eta_2\eta_3)}ds\\
&&+C_{p,T}\EE\int^T_0|Y^{\vare}_s-\hat Y^{\vare}_s|^{\eta_2}ds.
\end{eqnarray*}
By Lemmas \ref{COX} and \ref{DEY}, we get for any $\theta\in (0,2/\alpha)$,
$$
\EE\left(\sup_{t\in[0, T]}|X^{\vare}_t-\hat{X}^{\vare}_t|^p\right)\leq C_{p,T}(1+|x|)\delta^{\frac{\theta[\eta_1\wedge(\eta_2\eta_3)\wedge(\gamma+\theta'-1)]}{2}}.
$$
Hence, \eref{F4.9} holds by taking $\tilde{\theta}=\frac{\theta[\eta_1\wedge(\eta_2\eta_3)\wedge(\gamma+\theta'-1)]}{2}$. The proof is complete.
\end{proof}

\subsection{The main result }

Now, we give a position to state our main result.
\begin{theorem}\label{main result 1}
Assume that the assumptions \ref{A1}-\ref{A4} hold. Then for any $x, y\in H$, $p>0$ and $T>0$, we have
\begin{align}
\lim_{\vare\rightarrow 0}\mathbb{E} \left(\sup_{t\in[0,T]}|X_{t}^{\vare}-\bar{X}_{t}|^{p} \right)=0, \label{MR}
\end{align}
where $\bar{X}_t$ is the solution of the averaged equation \eref{3.1}.
\end{theorem}
\begin{proof}
Using \eref{FbarX} and \eref{AMX}, it is easy to see that
\begin{eqnarray*}
\hat{X}^{\vare}_{t}-\bar{X}_{t}=\!\!\!\!\!\!\!\!&&\int^t_0 e^{(t-s)A}\lambda (U_{\lambda}(X^{\vare}_s)-U_{\lambda}(\bar X_s))ds+U_{\lambda}(\bar{X}_t)-U_{\lambda}(X^{\vare}_t)\\
&&+\int^t_0 Ae^{(t-s)A}\left[U_{\lambda}(\bar{X}_{s})-U_{\lambda}(X^{\vare}_{s})\right]ds\\
&&+\sum_{k\geq 1}\int^t_0\!\int_{\RR} e^{(t-s)A}\Big[U_{\lambda}(X^\varepsilon_{s-}\!+\!\beta_k z e_k)-U_{\lambda}(X^\varepsilon_{s-})\\
&&\quad\quad\quad\quad\quad\quad-U_{\lambda}(\bar{X}_{s-}\!+\!\beta_k z e_k)\!+\!U_{\lambda}(\bar{X}_{s^{-}})\Big]\!\tilde{N}^{1,k}(ds,dz)\\
&&+\int^t_0 e^{(t-s(\delta))A}\langle DU_{\lambda}(X^{\vare}_{s(\delta)})+I, B(X^{\vare}_{s(\delta)},  \hat Y^{\vare}_s)-\bar{B}(X^{\vare}_{s(\delta)})\rangle ds.
\end{eqnarray*}
Then for any $p\geq 1$, we have the following estimate:
\begin{eqnarray}
&&\EE\left(\sup_{t\in[0, T]}|\hat{X}^{\vare}_{t}-\bar{X}_{t}|^p\right)\nonumber\\
\leq \!\!\!\!\!\!\!\!&& C_p\lambda^p T\EE\int^T_0 |U_{\lambda}(X^{\vare}_s)-U_{\lambda}(\bar X_s)|^pds+C_p\EE\left(\sup_{t\in[0, T]}|U_{\lambda}(X^{\vare}_t)-U_{\lambda}(\bar X_t)|^p\right)\nonumber\\
&&+C_p\EE\left(\sup_{t\in[0, T]}\left|\int^t_0 Ae^{(t-s)A}\Big(U_{\lambda}(X^{\vare}_s)-U_{\lambda}(\bar X_s)\Big)ds\right|^p\right)\nonumber\\
&&+C_p\EE\bigg(\sup_{t\in[0, T]}\bigg|\sum_{k\geq 1}\int^t_0\int_\mathbb{R} e^{(t-s)A}\big[U_{\lambda}(X^\varepsilon_{s^{-}}\!+\!\beta_k z e_k)\!-\!U_{\lambda}(X^\varepsilon_{s-})\nonumber\\
&&\quad\quad\quad\quad\quad\quad-U_{\lambda}(\bar{X}_{s-}\!+\!\beta_k z e_k)\!+\!U_{\lambda}(\bar{X}_{s-})\big]\tilde{N}^{1,k}(ds,dz)\bigg|^p\bigg)\nonumber\\
&&+C_p\EE\left(\sup_{t\in[0, T]}\left|\int^t_0 e^{(t-s(\delta))A}\langle DU_{\lambda}(X^{\vare}_{s(\delta)})+I, B(X^{\vare}_{s(\delta)},  \hat Y^{\vare}_s)-\bar{B}(X^{\vare}_{s(\delta)})\rangle ds\right|^p\right)\nonumber\\
:=\!\!\!\!\!\!\!\!&&\sum^5_{i=1}J_i(T).\label{3.26}
\end{eqnarray}

By \eref{BU} we have
\begin{eqnarray}
J_1(T)&\leq C_{p,T}\lambda^p \tilde{C}^p_{\lambda}\|\bar B\|^p_{C^{\eta_1\wedge(\eta_2\eta_3)}_b}\int^T_0\EE|X^{\vare}_s-\bar X_s|^p ds\label{J1}
\end{eqnarray}
and
\begin{eqnarray}
J_2(T)&\leq C_p \tilde{C}^p_{\lambda}\|\bar B\|^p_{C^{\eta_1\wedge(\eta_2\eta_3)}_b}\EE\left(\sup_{t\in [0,T]}|X^{\vare}_t-\bar X_t|^p\right).\label{J2}
\end{eqnarray}

Using the factorization method and taking a constant $\kappa_2\in(0,\kappa_1)$, where $\kappa_1$ is the one in the assumption \ref{A3}, we can write
$$
\int^t_0 Ae^{(t-s)A}(U_{\lambda}\left(X^{\vare}_s)-U_{\lambda}(\bar X_s)\right)ds=\frac{\sin( \pi\kappa_2)}{\pi}\int^t_0 e^{(t-s)A}(t-s)^{\kappa_2-1}f_sds,
$$
where $f_s:=\int^s_0 Ae^{(s-r)A}(s-r)^{-\kappa_2}(U_{\lambda}(X^{\vare}_r)-U_{\lambda}(\bar X_r))dr$. Taking $p$ large enough such that $\frac{p(1-\kappa_2)}{p-1}<1$, we have
\begin{eqnarray*}
\left|\int^t_0 Ae^{(t-s)A}(U_{\lambda}(X^{\vare}_s)-U_{\lambda}(\bar X_s))ds\right|\leq\!\!\!\!\!\!\!\!&&C\left(\int^t_0 (t-s)^{-\frac{p(1-\kappa_2)}{p-1}}ds\right)^{\frac{p-1}{p}}\|f\|_{L^{p}([0,T]; H)}\\
\leq\!\!\!\!\!\!\!\!&&C_{p}t^{\kappa_2-\frac{1}{p}}\|f\|_{L^{p}([0,T]; H)},
\end{eqnarray*}
where $\|f\|_{L^{p}([0,T]; H)}:=\left(\int^T_0 |f_t|^p_H dt\right)^{1/p}$. Then it follows
\begin{eqnarray*}
\sup_{0 \leq t\leq T}\left|\int^t_0 Ae^{(t-s)A}(U_{\lambda}(X^{\vare}_s)-U_{\lambda}(\bar X_s))ds\right|^{p}\leq\!\!\!\!\!\!\!\!&&C_{p,T}|f|^{p}_{L^{p}([0,T]; H)}.
\end{eqnarray*}
Choosing $p$ large enough such that $\frac{p(1-\kappa_1+\kappa_2)}{p-1}<1$, \eref{BU2} yields that
\begin{eqnarray}
J_3(T)\leq\!\!\!\!\!\!\!\!&&C_{p,T}\EE\int^T_0 |f_t|^p dt\nonumber\\
\leq\!\!\!\!\!\!\!\!&&C_{p,T}\EE\int^T_0\!\!\Big[\left(\int^t_0 \|(-A)^{1-\kappa_1}e^{(t-r)A}(t-r)^{-\kappa_2}\|^{\frac{p}{p-1}}dr\right)^{p-1}\nonumber\\
&&\quad\quad\quad\quad\quad \cdot\int^t_0 \|(-A)^{\kappa_1}DU_{\lambda}\|^p_{\infty} |X^{\vare}_r-\bar X_r|^p dr \Big]dt\nonumber\\
\leq\!\!\!\!\!\!\!\!&&C_{p,T}\|\bar{B}\|_{C^{\eta_1\wedge(\eta_2\eta_3)}_b}\left(\int^T_0 t^{-\frac{p(1-\kappa_1+\kappa_2)}{p-1}}dt\right)^{p-1}\int^T_0\EE|X^{\vare}_t-\bar{X}_t|^p dt\nonumber\\
\leq\!\!\!\!\!\!\!\!&&C_{p,T}\|\bar{B}\|_{C^{\eta_1\wedge(\eta_2\eta_3)}_b}\int^T_0\EE|X^{\vare}_t-\bar X_t|^pdt.\label{J3}
\end{eqnarray}

For the term $J_4(T)$, by Maximal inequality (see e.g., \cite[Proposition 3.3]{MPR}), we get
\begin{eqnarray*}
J_4(T)
\leq\!\!\!\!\!\!\!\!&&C_p\EE\int^T_0\sum_{k\geq 1}\Big[\int_\mathbb{R} |U_{\lambda}(X^\varepsilon_{s^{-}}\!+\!\beta_k z e_k)\!
-\!U_{\lambda}(X^\varepsilon_{s})-U_{\lambda}(\bar{X}_{s}\!+\!\beta_k z e_k)\!+\!U_{\lambda}(\bar{X}_{s})|^p\nu(dz)\\
&&\quad+\big(\int_\mathbb{R} |U_{\lambda}(X^\varepsilon_{s}\!+\!\beta_k z e_k)\!-\!U_{\lambda}(X^\varepsilon_{s-})-U_{\lambda}(\bar{X}_{s}\!+\!\beta_k z e_k)\!+\!U_{\lambda}(\bar{X}_{s})|^2\nu(dz)\big)^{p/2}
\Big]ds\\
=\!\!\!\!\!\!\!\!&&C_p\EE\int^T_0\sum_{k\geq 1}\Big[\beta^{\alpha}_k\int_\mathbb{R} |U_{\lambda}(X^\varepsilon_{s}\!+z e_k)\!
-\!U_{\lambda}(X^\varepsilon_{s})-U_{\lambda}(\bar{X}_{s}\!+ze_k)\!+\!U_{\lambda}(\bar{X}_{s})|^p\nu(dz)\\
&&\quad+\beta^{(\alpha p)/2}_k\big(\int_\mathbb{R} |U_{\lambda}(X^\varepsilon_{s}\!+z e_k)\!-\!U_{\lambda}(X^\varepsilon_{s})-U_{\lambda}(\bar{X}_{s}\!+z e_k)\!+\!U_{\lambda}(\bar{X}_{s})|^2\nu(dz)\big)^{p/2}
\Big]ds\\
\leq\!\!\!\!\!\!\!\!&&J_{41}(T)+J_{42}(T),
\end{eqnarray*}
where
\begin{eqnarray*}
J_{41}(T):=\!\!\!\!\!\!\!\!&&C_p\EE\int^T_0\sum_{k\geq 1}\Big[\beta^{\alpha}_k\int_{|z|>1}|U_{\lambda}(X^\varepsilon_{s}\!+z e_k)\!
-\!U_{\lambda}(X^\varepsilon_{s})-U_{\lambda}(\bar{X}_{s}\!+z e_k)\!+\!U_{\lambda}(\bar{X}_{s})|^p\nu(dz)\\
&&\quad+\beta^{(\alpha p)/2}_k\big(\int_{|z|> 1} |U_{\lambda}(X^\varepsilon_{s}\!+z e_k)\!-\!U_{\lambda}(X^\varepsilon_{s-})-U_{\lambda}(\bar{X}_{s}\!+z e_k)\!+\!U_{\lambda}(\bar{X}_{s})|^2\nu(dz)\big)^{p/2}\Big]ds,
\end{eqnarray*}
\begin{eqnarray*}
J_{42}(T):=\!\!\!\!\!\!\!\!&&C_p\EE\int^T_0\sum_{k\geq 1}\Big[\beta^{\alpha}_k\int_{|z|\leq1} |U_{\lambda}(X^\varepsilon_{s}\!+z e_k)\!
-\!U_{\lambda}(X^\varepsilon_{s})-U_{\lambda}(\bar{X}_{s}\!+ z e_k)\!+\!U_{\lambda}(\bar{X}_{s})|^p\nu(dz)\\
&&\quad+\beta^{(\alpha p)/2}_k\big(\int_{|z|\leq 1} |U_{\lambda}(X^\varepsilon_{s}\!+ z e_k)\!-\!U_{\lambda}(X^\varepsilon_{s})-U_{\lambda}(\bar{X}_{s}\!+z e_k)\!+\!U_{\lambda}(\bar{X}_{s})|^2\nu(dz)\big)^{p/2}\Big]ds.
\end{eqnarray*}
On one hand, by the fact that $\|DU_{\lambda}\|_{\infty}<\infty$, we have
\begin{eqnarray}
J_{41}(T):=\!\!\!\!\!\!\!\!&&C_p\EE\int^T_0\sum_{k\geq 1}\Big[\beta^{\alpha}_k|X^\varepsilon_{s}-\bar{X}_{s}|^p\nu(|z|>1)+\beta^{(\alpha p)/2}_k|X^\varepsilon_{s}-\bar{X}_{s}|^p\left[\nu(|z|>1)\right]^{p/2}
\Big]ds\nonumber\\
\leq\!\!\!\!\!\!\!\!&&C_p\sum_{k\geq 1}(\beta^{\alpha}_k+\beta^{(\alpha p)/2}_k)\int^T_0\EE|X^\varepsilon_{s}-\bar{X}_{s}|^p ds.\label{J41}
\end{eqnarray}
On the other hand, refer to \cite[Lemma 4.3]{SXX0}, we have for any $\theta'\in (0, \eta_1\wedge(\eta_2\eta_3))$,
\begin{eqnarray}
|U_{\lambda}(x+z)-U_{\lambda}(x)-U_{\lambda}(y+z)+U_{\lambda}(y)|\leq |x-y||z|^{\gamma+\theta'-1}\|U_{\lambda}\|_{\gamma+\theta'}.\label{POU}
\end{eqnarray}
By $\eta_1\wedge(\eta_2\eta_3)\in (1+\alpha/2-\gamma,1)$, we can choose $\theta'$ close enough to $\eta_1\wedge(\eta_2\eta_3)$ such that
\begin{eqnarray*}
p(\gamma+\theta'-1)>\alpha,\quad \forall p\geq 2.
\end{eqnarray*}
Then by \eref{BU} and \eref{POU}, we get
\begin{eqnarray}
J_{42}(T):=\!\!\!\!\!\!\!\!&&C_p\|\bar{B}\|^p_{C^{\eta_1\wedge(\eta_2\eta_3)}_b}\EE\int^T_0\sum_{k\geq 1}\Big[\beta^{\alpha}_k|X^\varepsilon_{s}-\bar{X}_{s}|^p\int_{|z|\leq 1}|z|^{p(\gamma+\theta'-1)}\nu(dz)\nonumber\\
&&\quad\quad\quad+\beta^{(\alpha p)/2}_k|X^\varepsilon_{s}-\bar{X}_{s}|^p\big(\int_{|z|\leq 1}|z|^{2(\gamma+\theta'-1)}\nu(dz)\big)^{p/2}
\Big]ds\nonumber\\
\leq\!\!\!\!\!\!\!\!&&C_p\|\bar{B}\|^p_{C^{\eta_1\wedge(\eta_2\eta_3)}_b}\sum_{k\geq 1}(\beta^{\alpha}_k+\beta^{(\alpha p)/2}_k)\int^T_0\EE|X^\varepsilon_{s}-\bar{X}_{s}|^p ds.\label{J42}
\end{eqnarray}
The assumption \ref{A2}, \eref{J41} and \eref{J42} imply that
\begin{eqnarray}
J_4(T)
\leq\!\!\!\!\!\!\!\!&&C_{p}\left[\|\bar{B}\|^p_{C^{\eta_1\wedge(\eta_2\eta_3)}_b}+1\right]\int^T_0\EE|X^\varepsilon_{s}-\bar{X}_{s}|^p ds.\label{J4}
\end{eqnarray}

For term $J_5(T)$. By Proposition \ref{Ergodicity} and Khasminkii's time discretization argument, we can easily obtain
\begin{eqnarray}
J_{5}(T)\leq\!\!\!\!\!\!\!\!&&C_{p,T}(1+|y|^{\eta_2})\left(\frac{\vare^2}{\delta^2}+\frac{\vare}{\delta}+\delta^{2}\right).\label{J5}
\end{eqnarray}
Note that the proof of \eref{J5} follows the same steps as in the proof of \cite[Step 2 of Theorem 2.3]{SXX1}, so we omit the detailed proof here.

Hence, for $p$ large enough, \eref{J1}-\eref{J3}, \eref{J4} and \eref{J5} yield that
\begin{align*}
\EE\left(\sup_{t\in[0,T]}|X^{\vare}_t-\bar X_t|^p\right)\leq & C_p\EE\left(\sup_{t\in[0,T]}|X^{\vare}_t-\hat X^{\vare}_t|^p\right)+C_p\EE\left(\sup_{t\in[0,T]}|\hat X^{\vare}_t-\bar X_t|^p\right)\\
\leq & C_{p} \tilde{C}^p_{\lambda}\|\bar B\|^p_{C^{\eta_1\wedge(\eta_2\eta_3)}_b}\EE\left(\sup_{t\in [0,T]}|X^{\vare}_t-\bar X_t|^p\right)\\
&+C_{p,T}(\lambda^p\tilde{C}^p_{\lambda}\|\bar{B}\|_{C^{\eta_1\wedge(\eta_2\eta_3)}_b}+1)\int^T_0\EE\left(\sup_{s\in[0,t]}|X^{\vare}_s-\bar X_s|^p\right)dt\\
&+C_p\EE\left(\sup_{t\in[0,T]}|X^{\vare}_t-\hat X^{\vare}_t|^p\right)+C_{p,T}(1+|y|^{\eta_2})\left(\frac{\vare^2}{\delta^2}+\frac{\vare}{\delta}+\delta^{2}\right).
\end{align*}
Since $\lim_{\lambda\rightarrow\infty}\tilde{C}_\lambda=0$ , taking $\lambda$ sufficient large such that $C_{p} \tilde{C}^p_{\lambda}\|\bar B\|^p_{C^{\eta_1\wedge(\eta_2\eta_3)}_b}\leq 1/2$,
then by estimate \eref{DEX}, we have
\begin{eqnarray*}
\EE\left(\sup_{t\in[0,T]}|X^{\vare}_t-\bar X_t|^p\right)\leq\!\!\!\!\!\!\!\!&&(\lambda^p\tilde{C}^p_{\lambda}\|\bar{B}\|_{C^{\eta_1\wedge(\eta_2\eta_3)}_b}+1)\int^T_0\EE\left(\sup_{s\in[0,t]}|X^{\vare}_s-\bar X_s|^p\right)dt\\
&&+C_{p,T}(1+|x|)\delta^{\tilde{\theta}}+C_{p,T}(1+|y|^{\eta_2})\left(\frac{\vare^2}{\delta^2}+\frac{\vare}{\delta}+\delta^{2}\right).
\end{eqnarray*}
By Gronwall's inequality and taking $\delta=\vare^{1/2}$, we final obtain
\begin{eqnarray*}
\EE\left(\sup_{t\in[0,T]}|X^{\vare}_t-\bar X_t|^p\right)\leq\!\!\!\!\!\!\!\!&&C_{p,T,\lambda}(1+|x|+|y|^{\eta_2})\left(\vare^{\tilde{\theta}/2}+\vare^{1/4}\right),
\end{eqnarray*}
which implies the main result. The proof is complete.
\end{proof}
\begin{remark}
Note that \eref{MR} holds for any $p>0$, which seems a little  strange due to the solutions $X_{t}^{\vare}$ and $\bar{X}_{t}$ only have finite $p$-th moment for $p\in (0,\alpha)$.
In fact, by the formulation of the mild solution, it is easy to see that
$$
X_{t}^{\vare}-\bar{X}_{t}=\int^t_0e^{(t-s)A}B(X^{\varepsilon}_s, Y^{\varepsilon}_s)ds-\int^t_0e^{(t-s)A}\bar{B}(\bar{X}_s)ds.
$$
Note that the boundedness of $B$ and $\bar{B}$ in $H$, it is east check that $X_{t}^{\vare}-\bar{X}_{t}$ is still bounded in $H$.
Consequently, $\mathbb{E} \left(\sup_{t\in[0,T]}|X_{t}^{\vare}-\bar{X}_{t}|^{p} \right)$ makes sense for any $p>0$. However, by the reason of the techniques used in the proof,
it is necessary to prove \eref{MR} for $p$ is large enough.
\end{remark}

\section{Application to example}\label{Sec Exs}

In this section we will apply our main result to establish the averaging principle for a class of slow-fast SPDEs with H\"{o}lder continuous coefficients driven by $\alpha$-stable process.
i.e., considering the following non-linear stochastic heat equation on $D=[0,\pi]$ with Dirichlet boundary assumptions:
\begin{equation}\left\{\begin{array}{l}\label{Ex}
\displaystyle
dX^{\varepsilon}(t,\xi)=\left[\Delta X^{\varepsilon}(t,\xi)
+B(X^{\varepsilon}(t,\cdot), Y^{\varepsilon}(t,\cdot))(\xi)\right]dt+dL(t,\xi),\vspace{2mm}\\
\displaystyle dY^{\varepsilon}(t,\xi)=\frac{1}{\varepsilon}\left[\Delta Y^{\varepsilon}(t,\xi)+F(X^{\varepsilon}(t,\cdot), Y^{\varepsilon}(t,\cdot))(\xi)\right]dt
+\frac{1}{\varepsilon^{1/\alpha}}dZ(t,\xi),\vspace{2mm}\\
X^{\vare}(t,\xi)=Y^{\vare}(t,\xi)=0, \quad t> 0,\quad \xi\in\partial D,\vspace{2mm}\\
X^{\vare}(0,\xi)=x(\xi),Y^{\vare}(0,\xi)=y(\xi)\quad \xi\in D, x,y\in H,\end{array}\right.
\end{equation}
where $\partial D$ the boundary of $D$, $L_t$ and $Z_t$ are two cylindrical $\alpha$-stable processes with $\alpha\in(3/2,2)$. The coefficients which are bounded and measurable satisfy for any $x_1,x_2,y_1,y_2\in H$,
$$|B(x_1, y_1)(\xi)-B(x_2, y_2)(\xi)|\leq C\left(|x_1(\xi)-x_2(\xi)|^{\eta_1} + |y_1(\xi)-y_2(\xi)|^{\eta_2}\right),$$
$$|F(x_1, y_1)(\xi)-F(x_2, y_2)(\xi)|\leq C|x_1(\xi)-x_2(\xi)|^{\eta_3}+ L_F|y_1(\xi)-y_2(\xi)|.$$

Put Laplacian $\Delta$ in $H:=L^2(D)$ (with Dirichlet boundary assumptions):
$$Ax=\Delta x,\quad x\in \mathscr{D}(A)=H^{2}(D) \cap H^1_0(D).$$
Then the operator $A$ is a self-adjoint operator and possesses a complete orthonormal system of eigenfunctions, namely
$$e_n(\xi)=(\sqrt{2/\pi})\sin(n\xi),\quad\xi\in[0,\pi], n\in \mathbb{N}_{+}.$$
The corresponding eigenvalues of $A$ are $-\lambda_n=-n^{2}$. It is easy to prove that the assumption \ref{A1} holds.

Set $\beta_{n}=\gamma_{n}:=C\lambda^{-r}_{n}$ ($n\in\mathbb{N}_{+}$) with $r\in\left(\frac{1}{2\alpha}, \frac{\alpha-1}{\alpha}\right)$. It is easy to see that
$$\sum_{n\geq1}\beta^{\alpha}_n<\infty, \sum_{n\geq1}\gamma^{\alpha}_n<\infty,\quad \sum_{n\geq1}\frac{1}{\lambda_n}<\infty.$$
Thus assumption \ref{A2} holds.

In this case, we have for any $\kappa_1\in (0, \frac{\alpha-\alpha r-1}{\alpha})$,
\begin{eqnarray*}
\Lambda_t= \sup\limits_{n\geq1}\frac{e^{-\lambda_n t}\lambda^{1/\alpha}_n}{\beta_n}\leq\frac{C}{t^{r+1/\alpha}},\quad \Lambda_{3,\kappa_1}(t)= \sup\limits_{n\geq1}\frac{e^{-\lambda_n t}\lambda^{\kappa_1+1/\alpha}_n}{\gamma_n}\leq\frac{C}{t^{r+\kappa_1+1/\alpha}},\quad t>0.
\end{eqnarray*}
Moreover, we can take $\gamma=\frac{\alpha}{\alpha r+1}\in(1,\alpha)$ to see that for any $\gamma'<\gamma$,
\begin{eqnarray*}
\int_0^\infty\e^{-\lambda t}\Lambda_t^{\gamma'} dt<\infty,\quad \int_0^\infty\e^{-\lambda t}\Lambda_{3,\kappa_1}(t) dt<\infty,
\end{eqnarray*}
which verity assumption \ref{A3}.

Taking H\"{o}lder indexes $\eta_1,\eta_2,\eta_3\in (0,1)$ satisfying $\eta_1\wedge(\eta_2\eta_3)\in \left(1-\frac{\alpha(1-\alpha r)}{2(\alpha r+1)},1\right)$ and $\lambda_1-L_F>0$, then assumption \ref{A4} holds.
Hence, by Theorem \ref{main result 1}, the slow component $X^{\vare}$ of the stochastic system \eref{Ex} strongly convergence to the solution $\bar{X}$ of the corresponding averaged equation.

\vspace{0.5cm}
\textbf{Acknowledgment}. This work was supported in part by the National
Natural Science Foundation of China (Grant No.~11771187, 11931004, 12090011), Graduate Research and Innovation Program in Jiangsu Province (Grant No.~KYCX$20_{-}2204$) and the Priority Academic Program Development of Jiangsu Higher Education Institutions.

\vspace{0.3cm}

\end{document}